\documentclass[11pt,reqno]{amsart}

\usepackage{amssymb,amsmath,amsthm,amsfonts}
\usepackage{bbm}
\usepackage{a4wide}

\usepackage{bookmark}
\usepackage{hyperref}
\hypersetup{pdfstartview={FitH}}

\newtheorem{theorem}{Theorem}

\numberwithin{equation}{section}

\begin{document}

\title{A variational restriction theorem}

\author[V. Kova\v{c}]{Vjekoslav Kova\v{c}}
\address{Vjekoslav Kova\v{c}, Department of Mathematics, Faculty of Science, University of Zagreb, Bijeni\v{c}ka cesta 30, 10000 Zagreb, Croatia}
\email{vjekovac@math.hr}

\author[D. Oliveira e Silva]{Diogo Oliveira e Silva}
\address{Diogo Oliveira e Silva, School of Mathematics, University of Birmingham, Edgbaston, Birmingham, B15 2TT, England}
\email{d.oliveiraesilva@bham.ac.uk}

%\date{\today}

\renewcommand{\subjclassname}{\textup{2020} Mathematics Subject Classification}

 % Fourier and Fourier-Stieltjes transforms and other transforms of Fourier type


\begin{abstract}
We establish variational estimates related to the problem of restricting the Fourier transform of a three-dimensional function to the two-dimensional Euclidean sphere. At the same time, we give a short survey of the recent field of maximal Fourier restriction theory.
\end{abstract}

\maketitle

%%%%%%%%%%%%%%%%%%%%%%%%%%%%%%%%%%%%%%%%%%%%%%%%%%%%%%%%%%%%%%%%%%%%%%%%%%%%%%%%%%%%%%%%%%%%%%%%%%%%%%%%%%%%%%%%%

\section{Introduction}

This short note serves primarily as a commentary on the very recent topic of pointwise estimates for the operator that restricts the Fourier transform to a hypersurface. We will concentrate exclusively on the case of the two-dimensional unit sphere $\mathbb{S}^2$ in three-dimensional Euclidean space $\mathbb{R}^3$. This both simplifies the exposition and enables the formulation of more general results.

Classical Fourier restriction theory seeks for {\it a priori} $\textup{L}^p$-estimates for the operator $f\mapsto \widehat{f}|_S$, where $S$ is a hypersurface in the Euclidean space.
In the case of $\mathbb{S}^2\subset\mathbb{R}^3$, the endpoint Tomas--Stein inequality \cite{S93:habook, T75:restr} reads as follows:
\begin{equation}\label{eq:restriction}
\big\| \widehat{f}\,\big\vert_{\mathbb{S}^2} \big\|_{\textup{L}^2(\mathbb{S}^2,\sigma)} \lesssim \|f\|_{\textup{L}^{4/3}(\mathbb{R}^3)}.
\end{equation}
Here $\sigma$ denotes the standard surface measure on $\mathbb{S}^2$.
It is well-known that $4/3$ is the largest exponent that can appear on the right-hand side of \eqref{eq:restriction} provided that we stick to the $\textup{L}^2$-norm on the left-hand side. Estimate \eqref{eq:restriction} enables the {\it Fourier restriction operator}
\begin{equation}\label{eq:fourrestrop}
\mathcal{R}\colon\textup{L}^{4/3}(\mathbb{R}^3)\to\textup{L}^2(\mathbb{S}^2,\sigma)
\end{equation}
to be defined via an approximation of identity argument as follows.
Fix a complex-valued Schwartz function $\chi$ such that $\int_{\mathbb{R}^3}\chi(x)\,\textup{d}x=1$,
and write $\chi_{\varepsilon}$ for the $\textup{L}^1$-normalized dilate of $\chi$, defined as
\[ \chi_{\varepsilon}(x):=\varepsilon^{-3}\chi(\varepsilon^{-1}x), \]
for $x\in\mathbb{R}^3$ and $\varepsilon>0$.
Given $f\in\textup{L}^{4/3}(\mathbb{R}^3)$, then, thanks to \eqref{eq:restriction}, $\mathcal{R}f$ can be defined as the limit
\[ \lim_{\varepsilon\to0^+}(\widehat{f}\ast\chi_{\varepsilon})\big\vert_{\mathbb{S}^2} \]
in the norm of the space $\textup{L}^2(\mathbb{S}^2,\sigma)$.

Maximal restriction theorems were recently inaugurated by M\"{u}ller, Ricci, and Wright \cite{MRW16:maxrestr}.
In that work, the authors considered general $\textup{C}^2$ planar curves with nonnegative signed curvature equipped with affine arclength measure, and established a maximal restriction theorem in the full range of exponents where the usual restriction estimate is known to hold.
Shortly thereafter, Vitturi \cite{V17:maxrestr} provided an elementary argument which leads to a partial generalization to higher-dimensional spheres. In $\mathbb{R}^3$, Vitturi's result covers the full Tomas--Stein range, whose endpoint estimate amounts to
\begin{equation}\label{eq:maxrestriction}
\Big\| \sup_{\varepsilon>0} \big| (\widehat{f}\ast\chi_{\varepsilon})(\omega) \big| \Big\|_{\textup{L}^2_{\omega}(\mathbb{S}^2,\sigma)} \lesssim_{\chi} \|f\|_{\textup{L}^{4/3}(\mathbb{R}^3)}.
\end{equation}
An easy consequence of \eqref{eq:maxrestriction} and of obvious convergence properties in the dense class of Schwartz functions is the fact that the limit
\begin{equation}\label{eq:pointrestrictionres}
\lim_{\varepsilon\to0^+}(\widehat{f}\ast\chi_{\varepsilon})(\omega)
\end{equation}
exists for each $f\in\textup{L}^{4/3}(\mathbb{R}^3)$ and for $\sigma$-almost every $\omega\in\mathbb{S}^2$.
This enables us to recover the operator \eqref{eq:fourrestrop} also in the pointwise sense, and not only in the $\textup{L}^2$-norm, which was the main motivation behind the paper \cite{MRW16:maxrestr}.

For the elegant proof of \eqref{eq:maxrestriction}, Vitturi \cite{V17:maxrestr} used the following equivalent, non-oscillatory reformulation of the ordinary restriction estimate \eqref{eq:restriction}:
\begin{equation}\label{eq:restriction2}
\Big| \int_{(\mathbb{S}^2)^2} g(\omega) \overline{g(\omega')} h(\omega-\omega') \,\textup{d}\sigma(\omega) \,\textup{d}\sigma(\omega') \Big|
\lesssim \|g\|_{\textup{L}^2(\mathbb{S}^2,\sigma)}^2 \|h\|_{\textup{L}^{2}(\mathbb{R}^3)}.
\end{equation}
The proof of the equivalence between \eqref{eq:restriction} and \eqref{eq:restriction2} amounts to passing to the adjoint operator (i.e.\@ to a Fourier extension estimate) and expanding out the $\textup{L}^4$-norm using Plancherel's identity. These steps make the choice of exponents $4/3$ and $2$ into the most convenient one. The advantage of the expanded adjoint formulation \eqref{eq:restriction2} is that one can easily insert the iterated maximal function operator in $h$, and simply invoke its boundedness on $\textup{L}^2(\mathbb{R}^3)$. We refer the reader to \cite{V17:maxrestr} for details. In some sense, we will be following a similar step below.

In this paper we quantify the pointwise convergence result \eqref{eq:pointrestrictionres} in terms of the so-called variational norms.
These were introduced by L\'{e}pingle \cite{L76:var} in the context of probability theory. Various modifications were then defined and used by Bourgain for numerous problems in harmonic analysis and ergodic theory; see for instance \cite{B89:pt}.
Recall that, given a function $a\colon(0,\infty)\to\mathbb{C}$ and an exponent $\varrho\in[1,\infty)$, the \emph{$\varrho$-variation seminorm} of $a$ is defined as
\[ \|a\|_{\widetilde{\textup{V}}^{\varrho}} := \sup_{\substack{m\in\mathbb{N}\\ 0<\varepsilon_0<\varepsilon_1<\cdots<\varepsilon_m}}
\Big( \sum_{j=1}^{m} |a(\varepsilon_{j})-a(\varepsilon_{j-1})|^{\varrho} \Big)^{1/\varrho}. \]
In order to turn it into a \emph{$\varrho$-variation norm}, we simply add the term $|a(\varepsilon_0)|^{\varrho}$ as follows:
\[ \|a\|_{\textup{V}^{\varrho}} := \sup_{\substack{m\in\mathbb{N}\\ 0<\varepsilon_0<\varepsilon_1<\cdots<\varepsilon_m}}
\Big( |a(\varepsilon_0)|^{\varrho} + \sum_{j=1}^{m} |a(\varepsilon_{j})-a(\varepsilon_{j-1})|^{\varrho} \Big)^{1/\varrho}. \]
Since this is not the standard convention in the literature, we make a distinction as we are going to use both quantities. Clearly, $\|a\|_{\textup{V}^{\varrho}}$ controls both $\sup_{\varepsilon>0}|a(\varepsilon)|$ and the number of ``jumps'' of $a(\varepsilon)$ as $\varepsilon\to0^+$ (and even as $\varepsilon\to\infty$, which is not the interesting case here).
A particular instance of the main result of this paper (which is Theorem~\ref{thm:maintheorem} below) is the following variational generalization of estimate \eqref{eq:maxrestriction},
\begin{equation}\label{eq:varrest}
\Big\| \big\| (\widehat{f}\ast\chi_{\varepsilon})(\omega) \big\|_{\textup{V}^{\varrho}_{\varepsilon}} \Big\|_{\textup{L}^2_{\omega}(\mathbb{S}^2,\sigma)} \lesssim_{\chi,\varrho} \|f\|_{\textup{L}^{4/3}(\mathbb{R}^3)},
\end{equation}
when $\varrho\in(2,\infty)$ and $f\in\textup{L}^{4/3}(\mathbb{R}^3)$. The reader can still consider $\chi$ to be a fixed Schwartz function, but we are just about to discuss more general possible choices.
Variational estimates like \eqref{eq:varrest} for various averages and truncations of integral operators have been extensively studied; the papers \cite{AJS98,BMSW18,CJRW00,JSW08:var,MSZ20:var} are just a sample from the available literature.
In addition to quantifying the mere convergence, such estimates establish convergence in the whole $\textup{L}^p$-space in an explicit and quantitative manner, without the need for pre-existing convergence results on a dense subspace.
Later in the text we will also use the \emph{biparameter $\varrho$-variation seminorm}, defined for a function of two variables $b\colon(0,\infty)\times(0,\infty)\to\mathbb{C}$ as
\[ \|b\|_{\widetilde{\textup{W}}^{\varrho}} := \sup_{\substack{m,n\in\mathbb{N}\\ 0<\varepsilon_0<\varepsilon_1<\cdots<\varepsilon_m\\ 0<\eta_0<\eta_1<\cdots<\eta_n}}
\Big( \sum_{\substack{1\leq j\leq m\\ 1\leq k\leq n}}  \big| b(\varepsilon_{j},\eta_{k})-b(\varepsilon_{j},\eta_{k-1})-b(\varepsilon_{j-1},\eta_{k})+b(\varepsilon_{j-1},\eta_{k-1}) \big|^{\varrho} \Big)^{1/\varrho}. \]

It is natural to consider more general averaging functions $\chi$; this has already been suggested (albeit somewhat implicitly) in the papers \cite{MRW16:maxrestr,V17:maxrestr}.
It is clear from the proof in \cite{V17:maxrestr} that the function $\chi$ does not need to be smooth. One can, for instance, take $\chi$ to be the $\textup{L}^1$-normalized indicator function of the unit ball in $\mathbb{R}^3$, in which case the $(\widehat{f}\ast\chi_{\varepsilon})(\omega)$ become the usual Hardy--Littlewood averages of $\widehat{f}$ over Euclidean balls $\textup{B}(\omega,\varepsilon)$,
\[ \frac{1}{|\textup{B}(\omega,\varepsilon)|} \int_{\textup{B}(\omega,\varepsilon)} \widehat{f}(y) \textup{d}y. \]
Moreover, Ramos \cite{Ra18} concluded that, for each $f\in\textup{L}^{p}(\mathbb{R}^3)$ and $1\leq p\leq 4/3$, almost every point on the sphere is a Lebesgue point of $\widehat{f}$, i.e.\@ for $\sigma$-a.e.\@ $\omega\in\mathbb{S}^2$ we have that
\[ \lim_{\varepsilon\to0^+} \frac{1}{|B(\omega,\varepsilon)|} \int_{B(\omega,\varepsilon)} \big|\widehat{f}(y) - (\mathcal{R}f)(y) \big| \,\textup{d}y = 0. \]
Prior to \cite{Ra18} this had been confirmed by Vitturi \cite{V17:maxrestr} for functions $f\in\textup{L}^{p}(\mathbb{R}^3)$, $1\leq p\leq 8/7$, who repeated the two-dimensional argument of M\"{u}ller, Ricci, and Wright \cite{MRW16:maxrestr}.

Subsequent papers \cite{Kov19:var} and \cite{Ra19}, which appeared after the first version of the present paper, generalize the averaging procedure even further, by convolving $\widehat{f}$ with certain averaging measures $\mu$.
In light of this more recent research, we now take the opportunity to both generalize \eqref{eq:varrest} and complete our short survey of maximal and variational Fourier restriction theories with papers that appeared in the meantime.
In what follows, $\mu$ will be a finite complex measure defined on the Borel sets in $\mathbb{R}^3$; its dilates $\mu_{\varepsilon}$ are now defined as
\[ \mu_{\varepsilon}(E):=\mu(\varepsilon^{-1}E), \]
for every Borel set $E\subseteq\mathbb{R}^3$ and $\varepsilon>0$. For reasons of elegance, one can additionally assume that $\mu$ is {\it normalized} by $\mu(\mathbb{R}^3)=1$ and that it is {\it even}, i.e., centrally symmetric with respect to the origin, which means that
\[ \mu(-E)=\mu(E), \]
for each Borel set $E\subseteq \mathbb{R}^3$.

In \cite{Kov19:var} one of the present authors showed analogues of \eqref{eq:maxrestriction} and \eqref{eq:varrest} when $\chi$ is replaced by a measure $\mu$ whose Fourier transform $\widehat{\mu}$ is $\textup{C}^\infty$ and satisfies the decay condition
\begin{equation}\label{eq:decaycond}
|\nabla\widehat{\mu}(x)| \lesssim (1+|x|)^{-1-\delta},
\end{equation}
for some $\delta>0$.
%Here and in what follows, $|v|$ denotes the Euclidean norm of a vector $v\in\mathbb{R}^3$.
It is interesting to make the following observation, which applies to the cases of two or three dimensions only, as things improve in higher dimensions. If one takes $\mu$ to be the normalized spherical measure, i.e.\@ $\mu=\sigma/\sigma(\mathbb{S}^2)$, then the decay of $|\nabla\widehat{\mu}(x)|$ as $|x|\to\infty$ is only $O(|x|^{-1})$; see \cite{AS92:hmf}. Consequently, the results from \cite{Kov19:var} do not apply. This was one of the sources of motivation for Ramos \cite{Ra19}, who reused Vitturi's argument from \cite{V17:maxrestr} to conclude that the maximal estimate
\[ \Big\| \sup_{\varepsilon>0} \big| (\widehat{f}\ast\mu_{\varepsilon})(\omega) \big| \Big\|_{\textup{L}^2_{\omega}(\mathbb{S}^2,\sigma)} \lesssim_{\mu} \|f\|_{\textup{L}^{4/3}(\mathbb{R}^3)} \]
holds as soon as the maximal operator $h\mapsto\sup_{\varepsilon>0}(h\ast\mu_{\varepsilon})$ is bounded on $\textup{L}^2(\mathbb{R}^3)$. Relating this to the work of Rubio de Francia \cite{RdF86:max}, he further deduced the following sufficient condition in terms of the decays of $\widehat{\mu}$ and $\nabla\widehat{\mu}$:
\[ |\widehat{\mu}(x)| \lesssim (1+|x|)^{-\alpha} \text{ and } |\nabla\widehat{\mu}(x)| \lesssim (1+|x|)^{-\beta}, \text{ with } \alpha+\beta>1. \]
This condition includes the spherical measure as it satisfies $|\widehat{\mu}(x)|=O(|x|^{-1})$ and $|\nabla\widehat{\mu}(x)|=O(|x|^{-1})$ as $|x|\to\infty$.

The main result of this note is a variational estimate which generalizes \eqref{eq:varrest} slightly beyond the previously covered cases of the averaging functions $\chi$ or measures $\mu$.

\begin{theorem}\label{thm:maintheorem}
Suppose that $\varrho\in(1,\infty)$, and that $\mu$ is a normalized, even complex measure defined on the Borel subsets of $\mathbb{R}^3$, and satisfying \emph{any} of the following three conditions:
\begin{itemize}
\item[(a)] $-x\cdot\nabla\widehat{\mu}(x)\geq 0$ for each $x\in\mathbb{R}^3$, and
\begin{equation}\label{eq:condvariation}
\Big\| \big\| (h\ast\mu_{\varepsilon})(x) \big\|_{\widetilde{\textup{V}}^{\varrho}_{\varepsilon}} \Big\|_{\textup{L}_x^{2}(\mathbb{R}^3)}
\lesssim_{\mu,\varrho} \|h\|_{\textup{L}^{2}(\mathbb{R}^3)}
\end{equation}
holds for each Schwartz function $h;$
\item[(b)] $\varrho\in(2,\infty)$, while $\widehat{\mu}$ is $\textup{C}^2$ and satisfies the decay condition \eqref{eq:decaycond}$;$
\item[(c)] the inequality
\begin{equation}\label{eq:conddouble}
\Big\| \big\| (h\ast\mu_{\varepsilon}\ast\overline{\mu}_{\eta})(x) \big\|_{\widetilde{\textup{W}}^{\varrho}_{\varepsilon,\eta}} \Big\|_{\textup{L}_x^{2}(\mathbb{R}^3)}
\lesssim_{\mu,\varrho} \|h\|_{\textup{L}^{2}(\mathbb{R}^3)}
\end{equation}
holds for each Schwartz function $h$.
\end{itemize}
Then, for each Schwartz function $f$, the following estimate holds:
\begin{equation}\label{eq:varrestmu}
\Big\| \big\| (\widehat{f}\ast\mu_{\varepsilon})(\omega) \big\|_{\textup{V}^{\varrho}_{\varepsilon}} \Big\|_{\textup{L}^2_{\omega}(\mathbb{S}^2,\sigma)} \lesssim_{\mu,\varrho} \|f\|_{\textup{L}^{4/3}(\mathbb{R}^3)}.
\end{equation}
\end{theorem}

Let us immediately clarify one minor technical issue. Theorem~\ref{thm:maintheorem} claims estimate \eqref{eq:varrestmu} for Schwartz functions $f$ only, but it immediately extends to all $f\in\textup{L}^{4/3}(\mathbb{R}^3)$ whenever $\mu$ is absolutely continuous with respect to Lebesgue measure. Otherwise, we could run into measurability issues on the left-hand side of \eqref{eq:varrestmu} for singular measures $\mu$.

Condition (a) in Theorem \ref{thm:maintheorem} is quite restrictive, but it is satisfied at least when $\varrho>2$ and the Radon--Nikodym density of $\mu$ is a radial Gaussian function. Indeed, if $\textup{d}\mu(x)=\alpha^3 e^{-\pi\alpha^2|x|^2}\,\textup{d}x$ for some $\alpha\in(0,\infty)$, then
\begin{equation}\label{eq:forGauss}
-x\cdot\nabla\widehat{\mu}(x) = 2\pi\alpha^{-2}|x|^2 e^{-\pi\alpha^{-2}|x|^2}
\end{equation}
is nonnegative, and \eqref{eq:condvariation} is a standard estimate by Bourgain \cite[Lemma 3.28]{B89:pt}. In fact, Bourgain formulated \eqref{eq:condvariation} for one-dimensional Schwartz averaging functions in \cite{B89:pt}, but the proof carries over to higher dimensions. Alternatively, one can invoke more general results from the subsequent literature, such as the work of Jones, Seeger, and Wright \cite{JSW08:var}, which covered higher-dimensional convolutions, more general dilation structures, and both strong and weak-type variational estimates in a range of $\textup{L}^p$-spaces.

Theorem~\ref{thm:maintheorem} under condition (b) was covered by the paper \cite{Kov19:var}, up to minor technicalities, such as the fact that here we do not need $\widehat{\mu}$ to be smoother than $\textup{C}^2$. However, \cite{Kov19:var} was concerned with more general surfaces and more general measures $\sigma$ on them, while here we are able to give a more direct proof that is specific to the sphere and to the stated choice of Lebesgue space exponents. In fact, as we have already noted, the present proof predates \cite{Kov19:var}.

Condition (c) above is somewhat artificial and difficult to verify, but we include it since the proof that it implies \eqref{eq:varrestmu} will be the most straightforward.

Maximal restriction estimates have found a nice application in the very recent work of Bilz \cite{B20}, who used them to show that there exists a compact set $E\subset\mathbb{R}^3$ of full Hausdorff dimension that does not allow any nontrivial a priori Fourier restriction estimates for any nontrivial Borel measure on $E$. We do not discuss the details here, but rather refer an interested reader to \cite{B20}.

\subsection{Notation}
If $A,B\colon X\to[0,\infty)$ are two functions (or functional expressions) such that,  for each $x\in X$,  $A(x)\leq CB(x)$ for some unimportant constant $0\leq C<\infty$, then we write $A(x)\lesssim B(x)$ or $A(x) = O(B(x))$. If the constant $C$ depends on a set of parameters $P$, we emphasize it notationally by writing $A(x)\lesssim_P B(x)$ or $A(x) = O_P(B(x))$.

We write a variable in the subscript of the letter denoting a function space whenever we need to emphasize with respect to which variable the corresponding (semi)norm is taken. For instance, $\|g(\omega)\|_{\textup{L}^p_\omega}$ can be written in place of $\|g\|_{\textup{L}^p}$, while $\|g(\varepsilon)\|_{\textup{V}^\varrho_\varepsilon}$ can be written in place of $\|g\|_{\textup{V}^\varrho}$, whenever the variables $\omega,\varepsilon$ need to be written explicitly.

The {\it Fourier transform} of a function $f\in\textup{L}^1(\mathbb{R}^3)$ is normalized as
\[ \widehat{f}(y) := \int_{\mathbb{R}^3} f(x) e^{-2\pi i x\cdot y} \textup{d}x, \]
for each $y\in\mathbb{R}^3$.
Here $x\cdot y$ denotes the standard scalar product of vectors $x,y\in\mathbb{R}^3$, and integration is performed with respect to the Lebesgue measure.
The map $f\mapsto\widehat{f}$ is then extended, as usual, by continuity to bounded linear operators $\textup{L}^p(\mathbb{R}^3)\to\textup{L}^{p'}(\mathbb{R}^3)$ for each $p\in(1,2]$ and $p'=p/(p-1)$.

More generally, the Fourier transform of a complex measure $\mu$ is the function $\widehat{\mu}$ defined as
\[ \widehat{\mu}(y) := \int_{\mathbb{R}^3} e^{-2\pi i x\cdot y} \textup{d}\mu(x), \]
for each $y\in\mathbb{R}^3$.

The set of complex-valued Schwartz functions on $\mathbb{R}^3$ will be denoted by $\mathcal{S}(\mathbb{R}^3)$.

The remainder of this paper is devoted to the proof of Theorem~\ref{thm:maintheorem}.

\section{Proof of Theorem~\ref{thm:maintheorem}}

We need to establish \eqref{eq:varrestmu} assuming any one of the three conditions from the statement of Theorem~\ref{thm:maintheorem}.
Let us start with condition (a).

\begin{proof}[Proof of Theorem~\ref{thm:maintheorem} assuming condition (a)]
Start by observing that
\[ \sup_{\varepsilon_0>0} \big|\widehat{f}\ast\mu_{\varepsilon_0}\big| \leq \big|\widehat{f}\ast\mu\big| + \sup_{\varepsilon_0>0} \big|\widehat{f}\ast(\mu_{\varepsilon_0}-\mu_1)\big|, \]
that $\widehat{f}\ast\mu = (f\widehat{\mu})\widehat{}\,$, and that the ordinary restriction estimate \eqref{eq:restriction} applies to $f\widehat{\mu}$ and yields
\[ \big\| \widehat{f}\ast\mu \big\|_{\textup{L}^2(\mathbb{S}^2,\sigma)} \lesssim \|f\widehat{\mu}\|_{\textup{L}^{4/3}(\mathbb{R}^3)} \lesssim_{\mu} \|f\|_{\textup{L}^{4/3}(\mathbb{R}^3)}. \]
Thus, inequality \eqref{eq:varrestmu} reduces to two applications of
\begin{equation}\label{eq:varrest0}
\Big\| \big\| (\widehat{f}\ast\mu_{\varepsilon})(\omega) \big\|_{\widetilde{\textup{V}}^{\varrho}_{\varepsilon}} \Big\|_{\textup{L}^2_{\omega}(\mathbb{S}^2,\sigma)} \lesssim_{\mu,\varrho} \|f\|_{\textup{L}^{4/3}(\mathbb{R}^3)},
\end{equation}
which we proceed to establish.
The desired estimate \eqref{eq:varrest0} unfolds as
\begin{equation}\label{eq:varrest1}
\bigg\| \sup_{\substack{m\in\mathbb{N}\\ 0<\varepsilon_0<\varepsilon_1<\cdots<\varepsilon_m}} \Big( \sum_{j=1}^{m} \big| \big(\widehat{f} \ast (\mu_{\varepsilon_{j-1}} - \mu_{\varepsilon_{j}}) \big) (\omega) \big|^\varrho\Big)^{1/\varrho} \bigg\|_{\textup{L}^2_{\omega}(\mathbb{S}^2,\sigma)}
\lesssim_{\mu,\varrho} \|f\|_{\textup{L}^{4/3}(\mathbb{R}^3)}.
\end{equation}
The numbers $\varepsilon_j$ in the above supremum can be restricted to a fixed interval $[\varepsilon_\textup{min},\varepsilon_\textup{max}]$ with $0<\varepsilon_\textup{min}<\varepsilon_\textup{max}$, but the estimate needs to be established with a constant independent of $\varepsilon_\textup{min}$ and $\varepsilon_\textup{max}$. Afterwards one simply applies the monotone convergence theorem letting $\varepsilon_\textup{min}\to0^+$ and $\varepsilon_\textup{max}\to\infty$. Moreover, by only increasing the left-hand side of \eqref{eq:varrest1}, we can also achieve $\varepsilon_0=\varepsilon_\textup{min}$ and $\varepsilon_m=\varepsilon_\textup{max}$.

Next, by continuity one may further restrict attention to rational numbers in the interval $[\varepsilon_\textup{min},\varepsilon_\textup{max}]$, and by yet another application of the monotone convergence theorem one may consider only finitely many values in that interval.
In this way, no generality is lost in assuming that the supremum in \eqref{eq:varrest1} is achieved for some $m\in\mathbb{N}$ and for some measurable functions $\varepsilon_k\colon\mathbb{S}^2\to[\varepsilon_\textup{min},\varepsilon_\textup{max}]$, $k\in\{0,1,\ldots,m\}$, such that $\varepsilon_0(\omega)\equiv\varepsilon_{\textup{min}}$, $\varepsilon_m(\omega)\equiv\varepsilon_{\textup{max}}$.
Estimate \eqref{eq:varrest1} then becomes
\[ \Big\| \Big( \sum_{j=1}^{m} \big| \big(\widehat{f} \ast (\mu_{\varepsilon_{j-1}(\omega)} - \mu_{\varepsilon_{j}(\omega)}) \big) (\omega) \big|^\varrho\Big)^{1/\varrho} \Big\|_{\textup{L}^2_{\omega}(\mathbb{S}^2,\sigma)} \lesssim_{\mu,\varrho} \|f\|_{\textup{L}^{4/3}(\mathbb{R}^3)}. \]
Once again, the implicit constant needs to be independent of $m$ and the functions $\varepsilon_k$. The reduction we just performed is an instance of the Kolmogorov--Seliverstov--Plessner linearization method used in \cite{MRW16:maxrestr}.

Dualizing the mixed $\textup{L}_\omega^2(\ell_j^\varrho)$-norm, see \cite{BP61}, we turn the latter estimate into
\begin{equation}\label{eq:varrest2}
\big| \Lambda(f,\mathbf{g}) \big| \lesssim_{\mu,\varrho} \|f\|_{\textup{L}^{4/3}(\mathbb{R}^3)} \Big\| \Big( \sum_{j=1}^{m} |g_j|^{\varrho'} \Big)^{1/\varrho'} \Big\|_{\textup{L}^2(\mathbb{S}^2,\sigma)},
\end{equation}
where the bilinear form $\Lambda$ is defined via
\[ \Lambda(f,\mathbf{g}) := \int_{\mathbb{S}^2} \sum_{j=1}^{m} \big(\widehat{f} \ast (\mu_{\varepsilon_{j-1}(\omega)} - \mu_{\varepsilon_{j}(\omega)}) \big) (\omega) \overline{g_j(\omega)} \,\textup{d}\sigma(\omega). \]
Here, $\varrho'=\varrho/(\varrho-1)$ denotes the exponent conjugate to $\varrho$ as usual, and $g_j\colon\mathbb{S}^2\to\mathbb{C}$ are arbitrary measurable functions, $j\in\{1,2,\ldots,m\}$, gathered in a single vector-valued function $\mathbf{g}=(g_j)_{j=1}^{m}$.
By elementary properties of the Fourier transform, $\Lambda$ can be rewritten as
\[ \Lambda(f,\mathbf{g}) = \int_{\mathbb{R}^3} f(x) \overline{\mathcal{E}(\mathbf{g})(x)} \textup{d}x, \]
where $\mathcal{E}$ is a certain extension-type operator given by
\begin{equation}\label{eq:defofE}
\mathcal{E}(\mathbf{g})(x) := \int_{\mathbb{S}^2} \sum_{j=1}^{m} \Big( \widehat{\mu}\big(\varepsilon_{j-1}(\omega)x\big) - \widehat{\mu}\big(\varepsilon_{j}(\omega)x\big) \Big) g_j(\omega) e^{2\pi i x\cdot\omega} \,\textup{d}\sigma(\omega).
\end{equation}
By H\"{o}lder's inequality, \eqref{eq:varrest2} is in turn equivalent to
\begin{equation}\label{eq:varrest3}
\| \mathcal{E}(\mathbf{g}) \|_{\textup{L}^{4}(\mathbb{R}^3)}
\lesssim_{\mu,\varrho} \Big\| \Big( \sum_{j=1}^{m} |g_j|^{\varrho'} \Big)^{1/\varrho'} \Big\|_{\textup{L}^2(\mathbb{S}^2,\sigma)}.
\end{equation}

If we denote
\begin{equation}\label{eq:defofvartheta}
\vartheta(x) := - x \cdot (\nabla\widehat{\mu})(x),
\end{equation}
then we also have
\[ - t \frac{\textup{d}}{\textup{d}t} \widehat{\mu}(tx) = -(tx) \cdot (\nabla\widehat{\mu})(tx) = \vartheta(tx), \]
for any $x\in\mathbb{R}^3$, which in turn implies
\begin{equation}\label{eq:fundthmcalc}
\widehat{\mu}(ax) - \widehat{\mu}(bx) = \int_{a}^{b} \vartheta(t x) \frac{\textup{d}t}{t},
\end{equation}
for any $0<a<b$.
Substituting this into the definition of $\mathcal{E}$, yields
\[ \mathcal{E}(\mathbf{g})(x) = \int_{\varepsilon_{\textup{min}}}^{\varepsilon_{\textup{max}}} \vartheta(tx) \int_{\mathbb{S}^2}
g_{j(t,\omega)}(\omega) e^{2\pi i x\cdot\omega} \,\textup{d}\sigma(\omega) \frac{\textup{d}t}{t}, \]
where, for each $t\in[\varepsilon_{\textup{min}},\varepsilon_{\textup{max}})$ and each $\omega\in\mathbb{S}^2$, we denote by $j(t,\omega)$ the unique index $j\in\{1,2,\ldots,m\}$ such that $t\in[\varepsilon_{j-1}(\omega),\varepsilon_{j}(\omega))$.
Since $\vartheta(x)\geq 0$, by the standing assumption in (a), we can apply the Cauchy--Schwarz inequality in the variable $t$ to estimate
\begin{equation}\label{eq:auxest1}
| \mathcal{E}(\mathbf{g})(x) |^2 \lesssim_{\mu} \mathcal{A}(x) \mathcal{B}(\mathbf{g})(x),
\end{equation}
where
\[ \mathcal{A}(x) := \int_{\varepsilon_{\textup{min}}}^{\varepsilon_{\textup{max}}} \vartheta(tx) \frac{\textup{d}t}{t} \]
and
\begin{equation}\label{eq:calbdef}
\mathcal{B}(\mathbf{g})(x) := \int_{\varepsilon_{\textup{min}}}^{\varepsilon_{\textup{max}}} \vartheta(tx) \Big| \int_{\mathbb{S}^2} g_{j(t,\omega)}(\omega) e^{2\pi i x\cdot\omega} \,\textup{d}\sigma(\omega) \Big|^2 \frac{\textup{d}t}{t}.
\end{equation}
By \eqref{eq:fundthmcalc},
\begin{equation}\label{eq:auxest2}
\mathcal{A}(x) = \widehat{\mu}\big(\varepsilon_{\textup{min}}x\big) - \widehat{\mu}\big(\varepsilon_{\textup{max}}x\big) \lesssim 1.
\end{equation}
From \eqref{eq:auxest1} and \eqref{eq:auxest2}, we see that \eqref{eq:varrest3} will be established once we prove
\begin{equation}\label{eq:varrest9}
\| \mathcal{B}(\mathbf{g}) \|_{\textup{L}^{2}(\mathbb{R}^3)}
\lesssim_{\mu,\varrho} \Big\| \Big( \sum_{j=1}^{m} |g_j|^{\varrho'} \Big)^{1/\varrho'} \Big\|_{\textup{L}^2(\mathbb{S}^2,\sigma)}^2 .
\end{equation}

Expanding out the square in the definition of $\mathcal{B}(\mathbf{g})(x)$ yields
\begin{equation}\label{eq:calb}
\mathcal{B}(\mathbf{g})(x)
= \int_{(\mathbb{S}^2)^2} \int_{\varepsilon_{\textup{min}}}^{\varepsilon_{\textup{max}}} g_{j(t,\omega)}(\omega) \overline{g_{j(t,\omega')}(\omega')} \,e^{2\pi i x\cdot(\omega-\omega')} \,\vartheta(tx) \frac{\textup{d}t}{t} \,\textup{d}\sigma(\omega) \,\textup{d}\sigma(\omega').
\end{equation}
For fixed $\omega,\omega'\in\mathbb{S}^2$, consider
\[ J(\omega,\omega') := \big\{ (j,j')\in\{1,2,\ldots,m\}^2 \,:\, [\varepsilon_{j-1}(\omega),\varepsilon_{j}(\omega)) \cap [\varepsilon_{j'-1}(\omega'),\varepsilon_{j'}(\omega')) \neq\emptyset \big\}. \]
The intersection of two half-open intervals is either the empty set or again a half-open interval.
For each pair $(j,j')\in J(\omega,\omega')$, it follows that there exist unique real numbers $a(j,j',\omega,\omega')$ and $b(j,j',\omega,\omega')$, such that
\begin{equation}\label{eq:intervals}
[\varepsilon_{j-1}(\omega), \varepsilon_{j}(\omega)) \cap [\varepsilon_{j'-1}(\omega'), \varepsilon_{j'}(\omega'))
= [a(j,j',\omega,\omega'), b(j,j',\omega,\omega')).
\end{equation}
Clearly the intervals \eqref{eq:intervals} constitute a finite partition of $[\varepsilon_{\textup{min}},\varepsilon_{\textup{max}})$. Using \eqref{eq:fundthmcalc}, we can rewrite \eqref{eq:calb} as
\begin{align*}
\mathcal{B}(\mathbf{g})(x) & = \int_{(\mathbb{S}^2)^2} \sum_{(j,j')\in J(\omega,\omega')} g_{j}(\omega) \overline{g_{j'}(\omega')}
\,e^{2\pi i x\cdot(\omega-\omega')}
\left(\int_{a(j,j',\omega,\omega')}^{b(j,j',\omega,\omega')} \vartheta(tx) \frac{\textup{d}t}{t} \right)\,\textup{d}\sigma(\omega) \,\textup{d}\sigma(\omega') \\
& = \int_{(\mathbb{S}^2)^2} \sum_{(j,j')\in J(\omega,\omega')} g_{j}(\omega) \overline{g_{j'}(\omega')}
\,e^{2\pi i x\cdot(\omega-\omega')} \\
& \qquad\qquad \times \Big( \widehat{\mu}\big(a(j,j',\omega,\omega')x\big) - \widehat{\mu}\big(b(j,j',\omega,\omega')x\big) \Big) \,\textup{d}\sigma(\omega) \,\textup{d}\sigma(\omega').
\end{align*}
Taking $h\in\mathcal{S}(\mathbb{R}^3)$ and dualizing with $\widehat{h}$ leads to the form
\[ \Theta(\mathbf{g},h) := \int_{\mathbb{R}^3} \mathcal{B}(\mathbf{g})(x) \widehat{h}(x) \,\textup{d}x. \]
By Plancherel's identity we have
\[ \| \mathcal{B}(\mathbf{g}) \|_{\textup{L}^{2}(\mathbb{R}^3)}
= \sup \big\{ |\Theta(\mathbf{g},h)| \,:\, h\in\mathcal{S}(\mathbb{R}^3),\ \|h\|_{\textup{L}^2(\mathbb{R}^3)}=1 \big\}, \]
so \eqref{eq:varrest9} will follow from
\begin{equation}\label{eq:varrest4}
| \Theta(\mathbf{g},h) | \lesssim_{\mu,\varrho} \Big\| \Big( \sum_{j=1}^{m} |g_j|^{\varrho'} \Big)^{1/\varrho'} \Big\|_{\textup{L}^2(\mathbb{S}^2,\sigma)}^2 \|h\|_{\textup{L}^{2}(\mathbb{R}^3)},
\end{equation}
which we proceed to establish.

Using basic properties of the Fourier transform, the form $\Theta$ can be rewritten as
\begin{align*}
\Theta(\mathbf{g},h)
& = \int_{(\mathbb{S}^2)^2} \sum_{(j,j')\in J(\omega,\omega')} g_{j}(\omega) \overline{g_{j'}(\omega')} \\
& \qquad\qquad \times \big(h\ast\mu_{a(j,j',\omega,\omega')} - h\ast\mu_{b(j,j',\omega,\omega')}\big)(\omega-\omega') \,\textup{d}\sigma(\omega) \,\textup{d}\sigma(\omega')
\end{align*}
for any Schwartz function $h$.
Applying H\"{o}lder's inequality to the sum in $(j,j')$, and recalling the definition of the $\varrho$-variation yields
\begin{align*}
|\Theta(\mathbf{g},h)| \leq \int_{(\mathbb{S}^2)^2}
 \Big(\sum_{j=1}^{m}|g_{j}(\omega)|^{\varrho'}\Big)^{1/\varrho'} \Big(\sum_{j=1}^{m}|g_{j}(\omega')|^{\varrho'}\Big)^{1/\varrho'}
 \| (h \ast \mu_\varepsilon)(\omega-\omega') \|_{\widetilde{\textup{V}}^\varrho_\varepsilon} \,\textup{d}\sigma(\omega) \,\textup{d}\sigma(\omega').
\end{align*}
By the usual Tomas--Stein  restriction theorem in the formulation \eqref{eq:restriction2}, applied with $g$ replaced by $\big(\sum_{j=1}^{m}|g_{j}|^{\varrho'}\big)^{1/\varrho'}$ and with $h$ replaced by $\|h\ast\mu_\varepsilon\|_{\widetilde{\textup{V}}^\varrho_\varepsilon}$, we obtain
\[ |\Theta(\mathbf{g},h)| \lesssim \Big\| \Big( \sum_{j=1}^{m} |g_j|^{\varrho'} \Big)^{1/\varrho'} \Big\|_{\textup{L}^2(\mathbb{S}^2,\sigma)}^2
\big\| \|h \ast \mu_\varepsilon\|_{\widetilde{\textup{V}}^\varrho_\varepsilon} \big\|_{\textup{L}^{2}(\mathbb{R}^3)}. \]
Invoking assumption \eqref{eq:condvariation} completes the proof of estimate \eqref{eq:varrest4}, and therefore also that of \eqref{eq:varrestmu}.
\end{proof}

Next, we will impose condition (b) and reduce the proof to the previous one by replacing $\mu$ with a superposition of ``nicer'' measures.

\begin{proof}[Proof of Theorem~\ref{thm:maintheorem} assuming condition (b)]
We can repeat the same steps as before, reducing \eqref{eq:varrestmu} to \eqref{eq:varrest9}, where $\mathcal{B}$ is as in \eqref{eq:calbdef} and $\vartheta$ is defined by \eqref{eq:defofvartheta}.

We have already noted that \eqref{eq:condvariation} is satisfied for measures $\mu$ with Gaussian densities, i.e., when $\textup{d}\mu(x)=\alpha^3 e^{-\pi \alpha^2|x|^2}\,\textup{d}x$ for some $\alpha\in(0,\infty)$, and that in this case \eqref{eq:forGauss} equals $\psi(x/\alpha)$,
where
\[ \psi(x) := 2\pi|x|^2 e^{-\pi|x|^2}. \]
Therefore, the previous proof of \eqref{eq:varrest9} specialized to this measure yields
\begin{equation}\label{eq:varrest8}
\bigg\| \int_{\varepsilon_{\textup{min}}}^{\varepsilon_{\textup{max}}} \psi\Big(\frac{tx}{\alpha}\Big) \Big| \int_{\mathbb{S}^2} g_{j(t,\omega)}(\omega) e^{2\pi i x\cdot\omega} \,\textup{d}\sigma(\omega) \Big|^2 \frac{\textup{d}t}{t} \bigg\|_{\textup{L}^{2}_{x}(\mathbb{R}^3)}
\lesssim_{\varrho} \Big\| \Big( \sum_{j=1}^{m} |g_j|^{\varrho'} \Big)^{1/\varrho'} \Big\|_{\textup{L}^2(\mathbb{S}^2,\sigma)}^2 .
\end{equation}
Also note that the $\textup{L}^1$-normalization of the above Gaussians guarantees that the left-hand side of \eqref{eq:condvariation} does not depend on the parameter $\alpha$. Consequently, the previous proof makes the constant in the bound \eqref{eq:varrest8} independent of $\alpha$ as well.
In this way, estimate \eqref{eq:varrest9} for a general measure $\mu$ satisfying condition (b) will be a consequence of \eqref{eq:varrest8} and Minkowski's inequality for integrals if we can only dominate $\vartheta(x)$ pointwise as follows:
\begin{equation}\label{eq:gausslemma3}
|\vartheta(x)| \lesssim_{\mu,\delta} \int_{1}^{\infty} \psi\Big(\frac{x}{\alpha}\Big) \frac{\textup{d}\alpha}{\alpha^{1+\delta}},
\end{equation}
for each $x\in\mathbb{R}^3$.

Denote by $\Psi(x)$ the right-hand side of \eqref{eq:gausslemma3}, and observe that $\Psi(0)=0$ and $\Psi(x)>0$ for each $x\neq 0$.
By continuity and compactness it suffices to show that the ratio $|\vartheta(x)|/\Psi(x)$ remains bounded as $|x|\to\infty$ or $|x|\to0^+$.

Substituting $r=\pi\alpha^{-2}|x|^2$, we may rewrite $\Psi$ as
\begin{align}
\Psi(x) & = 2\pi |x|^2 \int_{1}^{\infty} e^{-\pi\alpha^{-2} |x|^2} \frac{\textup{d}\alpha}{\alpha^{3+\delta}} \label{eq:gausspsi1} \\
& = \int_{0}^{\pi|x|^2} e^{-r} \Big(\frac{\sqrt{r}}{\sqrt{\pi}|x|}\Big)^{\delta} \,\textup{d}r. \label{eq:gausspsi2}
\end{align}
From \eqref{eq:gausspsi2} we see that $\lim_{|x|\to\infty} \Psi(x)/|x|^{-\delta} \in (0,\infty)$, while decay condition \eqref{eq:decaycond} gives
$|\vartheta(x)|=O(|x|^{-\delta})$ as $|x|\to\infty$.

On the other hand, using Taylor's formula for the function $x\mapsto e^{-\pi\alpha^{-2} |x|^2}$ and substituting into \eqref{eq:gausspsi1}, we easily obtain
\[ \Psi(x) = 2\pi |x|^2 \Big(\frac{1}{2+\delta} + O_{\delta}(|x|^2)\Big) \]
on a neighborhood of the origin. Moreover, $\widehat{\mu}$ is $\textup{C}^2$ and even, since $\mu$ is even, and so we have that $(\nabla\widehat{\mu})(0)=0$. Taylor's formula then yields
\[ \vartheta(x) = O_{\mu}(|x|^2) \]
on a neighborhood of the origin.
It follows that $|\vartheta(x)|/\Psi(x) = O_{\mu,\delta}(1)$ for sufficiently small nonzero $|x|$, and this completes the proof of \eqref{eq:gausslemma3}.
\end{proof}

The Gaussian domination trick which we have just used can be attributed to Stein, see \cite[Chapter~V, \S 3.1]{S70:sing}. It was generalized and used in a slightly different context by Durcik \cite{D15:L4entangled}.

\begin{proof}[Proof of Theorem~\ref{thm:maintheorem} assuming condition (c)]
We can repeat the same steps as before that reduce \eqref{eq:varrestmu} to \eqref{eq:varrest3}.
This time we define the form $\Theta$ differently, via
\[ \Theta(\mathbf{g},h) := \int_{\mathbb{R}^3} |\mathcal{E}(\mathbf{g})(x)|^2 \,\widehat{h}(x) \,\textup{d}x, \]
where $\mathbf{g}$ is as before and $h\in\mathcal{S}(\mathbb{R}^3)$.
Again, by duality, we only need to establish \eqref{eq:varrest4}.

Squaring out \eqref{eq:defofE} and inserting that into the above definition of $\Theta$, we obtain
\begin{align*}
\Theta(\mathbf{g},h)
& = \int_{\mathbb{R}^3} \int_{(\mathbb{S}^2)^2} \sum_{\substack{1\leq j\leq m\\ 1\leq k\leq m}} g_{j}(\omega) \overline{g_{k}(\omega')}
\,\Big( \widehat{\mu}\big(\varepsilon_{j-1}(\omega)x\big) - \widehat{\mu}\big(\varepsilon_{j}(\omega)x\big) \Big) \\
& \qquad\qquad\quad \times \Big( \overline{\widehat{\mu}\big(\varepsilon_{k-1}(\omega')x\big)} - \overline{\widehat{\mu}\big(\varepsilon_{k}(\omega')x\big)} \Big)
\,\widehat{h}(x) \,e^{2\pi i x\cdot(\omega-\omega')} \,\textup{d}\sigma(\omega) \,\textup{d}\sigma(\omega') \,\textup{d}x,
\end{align*}
i.e.,
\begin{align*}
\Theta(\mathbf{g},h)
& = \int_{(\mathbb{S}^2)^2} \sum_{\substack{1\leq j\leq m\\ 1\leq k\leq m}} g_{j}(\omega) \overline{g_{k}(\omega')} \\
& \qquad\qquad \times \big(h \ast ( \mu_{\varepsilon_{j-1}(\omega)} - \mu_{\varepsilon_{j}(\omega)} ) \ast \overline{( \mu_{\varepsilon_{k-1}(\omega)} - \mu_{\varepsilon_{k}(\omega)} )} \,\big) (\omega-\omega') \,\textup{d}\sigma(\omega) \,\textup{d}\sigma(\omega').
\end{align*}
Applying H\"{o}lder's inequality to the sum in $(j,k)$, and recalling the definition of the biparameter $\varrho$-variation seminorm yields
\begin{align*}
|\Theta(\mathbf{g},h)| \leq \int_{(\mathbb{S}^2)^2}
& \Big(\sum_{j=1}^{m}|g_{j}(\omega)|^{\varrho'}\Big)^{1/\varrho'} \Big(\sum_{k=1}^{m}|g_{k}(\omega')|^{\varrho'}\Big)^{1/\varrho'} \\
& \times \| (h \ast \mu_\varepsilon \ast \overline{\mu}_\eta)(\omega-\omega') \|_{\widetilde{\textup{W}}^\varrho_{\varepsilon,\eta}} \,\textup{d}\sigma(\omega) \,\textup{d}\sigma(\omega').
\end{align*}
Using \eqref{eq:restriction2} we obtain
\[ |\Theta(\mathbf{g},h)| \lesssim \Big\| \Big( \sum_{j=1}^{m} |g_j|^{\varrho'} \Big)^{1/\varrho'} \Big\|_{\textup{L}^2(\mathbb{S}^2,\sigma)}^2
\big\| \| h \ast \mu_\varepsilon \ast \overline{\mu}_\eta \|_{\widetilde{\textup{W}}^\varrho_{\varepsilon,\eta}} \big\|_{\textup{L}^{2}(\mathbb{R}^3)}, \]
and it remains to invoke the assumption \eqref{eq:conddouble}. This proves \eqref{eq:varrest4} and, thus, also completes the proof of the theorem assuming condition (c).
\end{proof}

\section*{Acknowledgments}
The authors are grateful to Polona Durcik and Jim Wright for useful discussions surrounding the problem, and to the anonymous referee for a careful reading of the manuscript and valuable suggestions.

V.K. was supported in part by the Croatian Science Foundation under the project UIP-2017-05-4129 (MUNHANAP).
He also acknowledges partial support of the DAAD--MZO bilateral grant \emph{Multilinear singular integrals and applications} and that of the \emph{Fulbright Scholar Program}.
D.O.S. acknowledges support by
the EPSRC New Investigator Award \emph{Sharp Fourier Restriction Theory}, grant no.\@ EP/T001364/1, and the College Early Career Travel Fund of the University of Birmingham.
This work was started during a pleasant visit of the second author to the University of Zagreb, whose hospitality is greatly appreciated.

%%%%%%%%%%%%%%%%%%%%%%%%%%%%%%%%%%%%%%%%%%%%%%%%


\begin{thebibliography}{99}

\bibitem{AS92:hmf} M. Abramowitz, I. A. Stegun (Eds.), \emph{Handbook of mathematical functions with formulas, graphs, and mathematical tables}, Dover Publications, Inc., New York, 1992.

\bibitem{AJS98} M. A. Akcoglu, R. L. Jones, P. O. Schwartz, \emph{Variation in probability, ergodic theory and analysis}, Illinois J. Math. {\bf 42} (1998), no. 1, 154--177.

\bibitem{BP61} A. Benedek, R. Panzone, \emph{The space $\textup{L}^p$, with mixed norm}, Duke Math. J. {\bf 28} (1961), 301--324.

\bibitem{B20} C. Bilz, \emph{Large sets without Fourier restriction theorems}, preprint (2020), available at arXiv:2001.10016.

\bibitem{B89:pt} J. Bourgain, \emph{Pointwise ergodic theorems for arithmetic sets}, with an appendix by the author, H. Furstenberg, Y. Katznelson, and D. S. Ornstein, Inst. Hautes \'{E}tudes Sci. Publ. Math. {\bf 69} (1989), 5--45.

\bibitem{BMSW18} J. Bourgain, M. Mirek, E. M. Stein, B. Wr\'{o}bel, \emph{On dimension-free variational inequalities for averaging operators in $\mathbb{R}^d$}, Geom. Funct. Anal. {\bf 28} (2018), no. 1, 58--99.

\bibitem{CJRW00} J. T. Campbell, R. L. Jones, K. Reinhold, M. Wierdl, \emph{Oscillation and variation for the Hilbert transform}, Duke Math. J. {\bf 105} (2000), no. 1, 59--83.

\bibitem{D15:L4entangled} P. Durcik, \emph{An $\textup{L}^4$ estimate for a singular entangled quadrilinear form}, Math. Res. Lett. {\bf 22} (2015), no. 5, 1317--1332.

\bibitem{JSW08:var} R. L. Jones, A. Seeger, J. Wright, \emph{Strong variational and jump inequalities in harmonic analysis}, Trans. Amer. Math. Soc. {\bf 360} (2008), no. 12, 6711--6742.

\bibitem{Kov19:var} V. Kova\v{c}, \emph{Fourier restriction implies maximal and variational Fourier restriction}, J. Funct. Anal. {\bf 277} (2019), no. 10, 3355--3372.

%\bibitem{KOS:v1} V. Kova\v{c}, D. Oliveira e Silva, \emph{A variational restriction theorem}, unpublished preprint (2018), available at arXiv:1809.09611v1.

\bibitem{L76:var} D. L\'{e}pingle, \emph{La variation d'ordre $p$ des semi-martingales}, Z. Wahrscheinlichkeitstheorie und Verw. Gebiete {\bf 36} (1976), no. 4, 295--316.

\bibitem{MSZ20:var} M. Mirek, E. M. Stein, P. Zorin-Kranich, \emph{A bootstrapping approach to jump inequalities and their applications}, Anal. PDE {\bf 13} (2020), no. 2, 527--558.

\bibitem{MRW16:maxrestr} D. M\"{u}ller, F. Ricci, J. Wright, \emph{A maximal restriction theorem and Lebesgue points of functions in $\mathcal{F}(\textup{L}^p)$}, Rev. Mat. Iberoam. \textbf{35} (2019), no. 3, 693--702.

\bibitem{Ra19} J. P. G. Ramos, \emph{Low-dimensional maximal restriction principles for the Fourier transform}, accepted for publication in Indiana Univ. Math. J., available at arXiv:1904.10858.

\bibitem{Ra18} J. P. G. Ramos, \emph{Maximal restriction estimates and the maximal function of the Fourier transform}, Proc. Amer. Math. Soc. {\bf 148} (2020), no. 3, 1131--1138.

\bibitem{RdF86:max} J. L. Rubio de Francia, \emph{Maximal functions and Fourier transforms}, Duke Math. J. \textbf{53} (1986), no. 2, 395--404.

\bibitem{S93:habook} E. M. Stein, \emph{Harmonic analysis: real-variable methods, orthogonality, and oscillatory integrals}, Princeton Math. Ser. \textbf{43}, Princeton Univ. Press, Princeton, 1993.

\bibitem{S70:sing} E. M. Stein, \emph{Singular integrals and differentiability properties of functions}, Princeton Math. Ser. \textbf{30}, Princeton Univ. Press, Princeton, 1970.

\bibitem{T75:restr} P. Tomas, \emph{A restriction theorem for the Fourier transform}, Bull. Amer. Math. Soc. \textbf{81} (1975), no. 2, 477--478.

\bibitem{V17:maxrestr} M. Vitturi, \emph{A note on maximal Fourier Restriction for spheres in all dimensions}, preprint (2017), available at arXiv:1703.09495.

\end{thebibliography}
\end{document}